\documentclass[11pt,final]{amsart}
\usepackage[pctex32]{graphics}
\usepackage{amsfonts}
\usepackage{amssymb,amscd,latexsym}
\usepackage{amsmath}
\usepackage{epsfig}
\usepackage{color}
\usepackage{color}

\usepackage[dvipsnames]{xcolor}
\usepackage{fancyvrb}   
 
\definecolor{airforceblue}{rgb}{0.36, 0.54, 0.66}
\definecolor{bleudefrance}{rgb}{0.19, 0.55, 0.91}
\definecolor{darkorchid}{rgb}{0.6, 0.2, 0.8}
\definecolor{darkorange}{rgb}{1.0, 0.55, 0.0}
\definecolor{darkspringgreen}{rgb}{0.09, 0.45, 0.27}

\usepackage[arrow,curve,matrix,tips,frame]{xy}

\usepackage[normalem]{ulem}
\usepackage[utf8]{inputenc}
\usepackage[english]{babel} 
\usepackage{enumerate}
\usepackage{hyperref}  
                 \hypersetup{ pdfborder={0 0 0}, 
                              colorlinks=true, 
                              linkcolor=blue, 
                              urlcolor=blue,
                              citecolor=blue,
                              linktoc=page,
                              pdfauthor={Francesco Russo, Giovanni Staglian{\`o}}, 
                              pdftitle={Congruences of 5-secant conics and the rationality of some admissible cubic fourfolds}
                            }                         

\usepackage{verbatim}

\newtheorem{thm}{Theorem}

\newtheorem{rmk}[thm]{Remark}

\theoremstyle{definition}

\newcommand{\PP}{{\mathbb{P}}}

\newcommand{\map}{\dasharrow}

\def\p{\mathbb P}

\def\I{\mathcal I}

\newcommand{\Sec}{\operatorname{Sec}}

\def\PP{{\mathbb P}}

\let\phi=\varphi

\newtheorem*{thmn0}{Theorem}

\begin{document}

\title[5-secant conics and the rationality of admissible cubic fourfolds]{Congruences of  5-secant conics and the rationality of some admissible cubic fourfolds}

\author[F. Russo]{Francesco Russo*}
\address{Dipartimento di Matematica e Informatica, Universit\` a degli Studi di Catania, Viale A. Doria 5, 95125 Catania, Italy}
\email{frusso@dmi.unict.it, \:giovannistagliano@gmail.com}
\thanks{*Partially  supported  by the PRIN {\it Geometria delle variet\`{a} algebriche} and by the FIR2014  {\it Aspetti geometrici e algebrici della Weak e Strong Lefschetz Property} of the University of Catania; the author is a member of the G.N.S.A.G.A. of INDAM}
\author[G. Staglian\` o]{Giovanni Staglian\` o}

\begin{abstract} The works of Hassett and Kuznetsov identify countably many divisors $C_d$ 
in the open subset of   $\p^{55}=\p(H^0(\mathcal O_{\p^5}(3)))$ parametrizing all cubic 4-folds and conjecture that 
the cubics corresponding to these divisors are precisely the rational 
ones. Rationality has been known classically for the first family $C_{14}$. We use 
congruences of 5-secant conics to prove rationality for the first three of the 
families $C_d$, corresponding to  $d=14, 26, 38$ in Hassett's notation.	
\end{abstract}

\maketitle

\section*{Introduction}
One of the most challenging open problems in classical and modern algebraic geometry is the rationality
of smooth cubic hypersurfaces $X\subset\p^5$. The very general cubic fourfold is expected to be irrational although
no such example is known.
Recent work of Hassett via Hodge Theory in   \cite{Hassett, Has00} (see also \cite{Levico}) and of Kuznetzsov via derived categories in  \cite{kuz4fold,kuz2}  lead  to the definition of infinitely many irreducible
divisors $\mathcal C_d$  in the moduli space $\mathcal C$ of cubic fourfolds corresponding to the  {\it admissible values} $d\in\mathbb N$ and to the conjecture that the cubics belonging to the admissible $\mathcal C_d$ should be precisely the rational ones ({\it Kuznetsov Conjecture}), see \cite{AT, kuz4fold, kuz2, Levico}. The admissible values are the even integers $d>6$  not divisible by 4, by 9 and nor by any odd prime of the form $2+3m$ (see for example \cite[Section 3]{Levico}), so that  the first admissible values are $d=14, 26, 38, 42$. Fano showed the
rationality of a general $[X]\in\mathcal C_{14}$, see \cite{Fano, BRS, Morin} and Section \ref{C14} for a different proof. Our  main result is the following:

\begin{thmn0}\label{intro-rat}
	Every cubic fourfold  in the irreducible divisors $\mathcal C_{26}$ and $\mathcal C_{38}$  is rational.
	\end{thmn0}

The Hodge theoretic definition of the
divisors $\mathcal C_d$ reduces to the existence of a rank two lattice $\langle h^2, S\rangle\subseteq H^{2,2}(X,\mathbb Z)$
of discriminant $d$, where $h$ is the class of a hyperplane section of $X$ and $S$ is the class of an algebraic
2-cycle on $X$. The elements of $\mathcal C_d$ are  called {\it special cubic fourfolds} and $\mathcal C_d\neq\emptyset$ if and only if $d>6$ and $d\equiv 0,2$ (mod 6), see \cite{Hassett, Has00}. 
The geometrical definition of $\mathcal C_d$
for small values of $d$ is by means of a particular surface $S_d\subset X$, which is
obviously not unique. For example, $\mathcal C_8$ is usually  described as the locus of smooth cubic
fourfolds containing a plane, while $\mathcal C_{14}$ can be described either as the closure of the locus
of cubic fourfolds containing a smooth quartic rational normal scroll  or as the closure
of the locus of those containing a smooth quintic del Pezzo surface, the so called Pfaffian locus. In 
\cite{Nuer} there are similar descriptions for 
the
values  
$12\leq d\leq 44$, $d\neq 42$, while for $d=42$ one can consult  \cite{Lai}.

Every $[X]\in\mathcal C_{14}$ has been proved to be rational by showing that it contains either  a smooth surface with one apparent double point or one of its
{\it small} degenerations, see \cite{BRS} and also Theorem \ref{famconics14} here for a different proof. The extension of  this geometrical approach to rationality for other (admissible) values $d$ appeared to be impossible  due to the paucity of surfaces with one apparent double point, which essentially can be used only for $d=14$.

 Our   discovery is the existence of irreducible
surfaces  $S_d\subset\p^5$ contained in a general element of $\mathcal C_d$, for $d$=14, 26 or 38, admitting a four-dimensional irreducible family of 5-secant conics such that
through a general point of $\p^5$ there passes a unique conic of the family. We dubbed these families  {\it congruences of
5-secant conics to $S_d$}. Once established the existence of such a congruence we  deduce that a general  cubic in $|H^0(\I_{S_d}(3))|$  is a rational section of the universal family of the congruence of 5-secant conics. In particular, such a general cubic is birational
to the parameter space of the congruence.
Then the rationality of the parameter space of the congruence (or the existence of a particular, sufficiently general, singular rational cubic hypersurface through $S_d$) shows that every smooth  cubic in $|H^0(\I_{S_d}(3))|$ is rational (and hence that a general $X\in\mathcal C_d$ is rational); see Theorem \ref{criterion} for a precise and general formulation of this principle. 

For $d=26$, which was the first open case, we take $S_{26}\subset\p^5$ to be the surface with one node  obtained as the projection of a smooth del Pezzo surface $S\subset\p^7$ of degree seven from a line intersecting the secant variety of $S$ transversally.  For $d=14$, we use  smooth projections of  rational octic  surfaces in $\p^6$ of sectional genus 3,
which are  the birational images of
$\p^2$ via  the linear system of plane quartic curves with eight general base points. For $d=38$, we consider smooth surfaces of degree 10 in $\p^5$ of sectional genus 6,  which are the birational images of $\p^2$ via the linear system of plane curves of degree 10 having ten general  base points of multiplicity three and which were  also used by Nuer in \cite{Nuer} to describe some birational properties of the divisor $\mathcal C_{38}$.

As far as we know,  the existence of   congruences  of 5-secant conics to surfaces in $\p^5$ has never been considered before,  even in the classical literature.
The surfaces $S_d\subset\p^5$ with $d=14, 26, 38$ we choose can be analyzed better nowadays via computational systems revealing remarkable algebraic properties of their defining equations. 

The  classification of irreducible surfaces in $\p^5$ admitting a congruence of $5$-secant conics is
an important and pressing open question we shall consider elsewhere. An obvious necessary condition is that the surface
is not contained in a quadric  hypersurface. This restriction plays a key role for the existence  of these surfaces and it explains some geometrical aspects of the paucity of examples. 

Our fundamental tool for determining the existence of  congruences of 5-secant conics has been the Hilbert scheme of lines passing through a general point
of a projective variety and contained in the variety; see  also \cite{libro} for other applications. The study of these lines for suitable birational images of $\p^5$ via the cubic equations
defining the {\it right} $S_d\subset X$ revealed the existence of the congruences of $5$-secant conics to $S_d$, which a posteriori have also
 natural geometrical constructions. 
To analyze the geometry of the images
of $\p^5$ via the linear system of cubics defining $S_d\subset\p^5$ we used  
Macaulay2 \cite{macaulay2}. 
The necessary ancillary files containing the scripts to verify some geometrical properties claimed throughout the paper are  posted on arXiv; see also Section \ref{comp} and also \cite{Schreyer} for the philosophy behind these verifications for random surfaces or hypersurfaces (although we work over $\mathbb Q$ and not over finite fields).

So far  $\mathcal C_{14}$, $\mathcal C_ {26}$ and $\mathcal C_{38}$ are the only loci of codimension one in $\mathcal C$ whose elements are known to be rational.  There exist infinitely many other irreducible loci of codimension at least two in $\mathcal C$ (some of them contained  also in non admissible $\mathcal C_d$'s), whose general element is rational, see \cite{Hassett} and \cite{AHTVA}.   
The next admissible value $d=42$  seems to deserve some special attention and it might be a watershed in order to confirm the validity of the Kuznetsov Conjecture also for the countably many admissible values $d>38$.
\medskip

{\bf Acknowledgements}. We wish to thank Ciro Ciliberto for useful and stimulating discussions and the referee's for some suggestions  about  the presentation of the results.

\section{Rationality via congruences of (\texorpdfstring{$re-1$}{re-1})-secant curves of degree \texorpdfstring{$e$}{e}}
Let $S\subset\p^5$ be an irreducible surface and let $\mathcal H$ be 
an irreducible proper family of  (rational or of fixed arithmetic genus) curves of degree $e$ in $\p^5$ whose general element is irreducible.
Let 
$$\pi:\mathcal D\to \mathcal H$$
be the universal family over $\mathcal H$ and let $$\psi:\mathcal D\to\p^5$$ be the tautological morphism. 
Suppose moreover that $\psi$ is birational and that a general member $[C]\in \mathcal H$ is ($re-1$)-secant to $S$,
that is $C\cap S$ is a length $r e-1$ scheme, $r\in\mathbb N$. We shall call such a family a {\it congruence of }  ($re-1$)-{\it secant curves of degree $e$ to $S$}. Let us remark that necessarily $\dim(\mathcal H)=4$.
\medskip

The  classification of irreducible surfaces in $\p^5$ admitting a congruence of $(re-1)$-secant rational curves of  degree $e\geq 2$ (and with $r\geq 3$) is
an open problem, never considered before. In the sequel we shall construct some surfaces admitting a congruence of 5-secant conics. We also know some 
 irreducible surfaces in $\p^5$ admitting a congruence of 8-secant twisted cubics and others admitting a congruence of 14-secant quintic rational normal curves.
\medskip

Let $S\subset\p^5$ be a surface admitting a congruence of {\rm(}$re-1${\rm)}-secant curves of degree $e$ parametrized by $\mathcal H$ and let notation be as above. The birationality of $\psi:\mathcal D\to\p^5$ and Zariski Main Theorem yield
that the locus of points $q\in\p^5$ through which there pass at least two ($re-1)$-secant curves of degree $e$  in the family $\mathcal H$  has codimension at least two in $\p^5$. Let $X\in|H^0(\mathcal I_{S}(r))|$ be fixed and irreducible and let $p\in X$ be a general point. Then through $p$  there passes a unique curve $C_p$ of the family $\mathcal H$. An irreducible  hypersurface $X\in|H^0(\mathcal I_{S}(r))|$ is said to be {\it transversal to the congruence $\mathcal H$} if  the unique curve of the congruence passing through a general point  $p\in X$ is
not contained in $X$. The next result will be crucial for our analysis. 

\begin{thm}\label{criterion} Let $S\subset\p^5$ be a surface admitting a congruence of {\rm(}$re-1${\rm)}-secant curves of degree $e$ parametrized by $\mathcal H$.
If  $X\in|H^0(\mathcal I_{S}(r))|$ is an irreducible hypersurface transversal to $\mathcal H$,  then $X$ is birational to $\mathcal H$.

If the map   
$\Phi=\Phi_{|H^0(\mathcal I_{S}(r))|}:\p^5\map \p(H^0(\mathcal I_{S}(r)))$
is birational onto its image, then a general  hypersurface  $X\in|H^0(\mathcal I_{S}(r))|$  is birational to $\mathcal H$.

Moreover, under the previous hypothesis on $\Phi$, if  a general  element in $|H^0(\mathcal I_{S}(r))|$ is smooth, then every $X\in |H^0(\mathcal I_{S}(r))|$ with at worst rational singularities is birational to $\mathcal H$.
\end{thm} 
\begin{proof} 
Let notation be as above and suppose that  the irreducible hypersurface $X\in|H^0(\mathcal I_{S}(r))|$ is transversal to $\mathcal H$. Then, for $p\in X$ general,  the curve $C_p$ is  not contained in $X$ and it cuts  $X$ outside $S$ only in the point $p$ by B\' ezout Theorem. 
Thus  $X$ naturally becomes 
a rational section of the family $\pi:\mathcal D\to\mathcal H$. Indeed, we get a rational map $\eta: X\map \mathcal D\stackrel{\pi}\to\mathcal H$ by associating to a general point $p\in X$ the same point $p\in C_p$
on the unique curve  $C_p=\mathcal D_{\pi(p)}$ passing through $p$ and belonging to the family $\mathcal H$. The rational map $\varphi:\mathcal H\map X$ is defined by taking  $[C]\in\mathcal H$
and setting $\phi([C])=C\cap (X\setminus S)$ to be the unique point in $X\cap C$ outside $S$.  For $p\in X$ general the map  $\phi$ is defined at $\eta(p)$
by the transversality hypothesis, $\phi(\eta(p))=p$ and $\eta$ is birational, as claimed.

Suppose that the map $\Phi:\p^5\map \p(H^0(\mathcal I_{S}(r)))$
is birational onto its image. Then, for $p\in \p^5 $ general, the curve $\overline{\Phi(C_p)}$ is a line $L_p$ through  $\Phi(p)$ and contained in $\overline{\Phi(\p^5)}$. A general  hypersurface $X\in |H^0(\mathcal I_{S}(r))|$ passing through $p$ is sent into a general hyperplane section $H$ of the image passing through $\Phi(p)$. Thus $H$  does not contain $L_p$ and  a general $X\in |H^0(\mathcal I_{S}(r))|$, which is irreducible, is transversal to $\mathcal H$ and hence birational to $\mathcal H$ by the first part.
Assume also that a general member of $|H^0(\mathcal I_{S}(r))|$ is smooth. Then every $X\in |H^0(\mathcal I_{S}(r))|$ with at worst rational singularities is birational to $\mathcal H$ by \cite[Theorem 1 and Theorem 16]{KT}.
\end{proof}

Until now the above construction has been applied only for $e=1$ and $r=3$ to show that a smooth  cubic hypersurface $X\subset\p^5$
containing a surface $S\subset \p^5$  which admits a congruence of secant lines, the so called  {\it surfaces with one apparent double point}, is rational. 
Indeed,  such irreducible surfaces $S\subset\p^5$ are known to be rational,
$\mathcal H$ is birational to the rational fourfold $S^{(2)}$ (see \cite{Mattuck}) and cubics through such a  $S$ satisfy the hypothesis in the second part of Theorem \ref{criterion}. 

In the sequel we shall apply Theorem \ref{criterion} for $e=2$ and $r=3$ to a general element of $\mathcal C_d$ with $d=14, 26, 38$, showing in different ways the rationality of $\mathcal H$. 

\section{Rationality of cubics in \texorpdfstring{$\mathcal C_{14}$}{C14} via congruences of 5-secant conics}\label{C14} 

Let  $S_{14}\subset\p^5$ be a smooth isomorphic projection of an  octic  smooth
surface $S\subset\p^6$ of sectional genus 3 in $\p^6$, obtained as the image of $\p^2$ via the linear system of quartic
curves with 8  general base points. 

\begin{thm}\label{famconics14} A  surface $S_{14}\subset\p^5$ as above admits a congruence of 5-secant conics parametrized by a rational variety birational to $S_{14}^{(2)}$. Moreover, every 
$[X]\in\mathcal C_{14}$ is rational.
\end{thm}
\begin{proof}
Let $S\subset\p^6$ be an octic surface of sectional genus 3 as described above. The surface $S$ has ideal generated by seven quadratic forms defining
a Cremona transformation:
$$\psi:\p^6\map \p^6,$$
whose base locus scheme is exactly $S$ and whose inverse is defined by forms of degree four, see \cite{ST, HKS}. The secant variety to $S$, $\Sec(S)\subset\p^6$, is an irreducible hypersurface
of degree 7 by the double point formula (see \cite[Theorem 9.3]{Fulton}) and it  is also the  exceptional 
locus 
of $\psi$, that is the closure
of the locus of points where $\psi$ is not an isomorphism. Thus $\Sec(S)$  is contracted to a four dimensional irreducible variety $W\subset\p^6$, which has degree 8 and which is 
 the base locus scheme of $\psi^{-1}$ (see \cite{ST,HKS}). The variety $W$ birationally parametrizes the secant lines to $S$ and it is thus  birational to the rational variety $S^{(2)}$.
 
Let $q\in\p^6\setminus\Sec(S)$ be general, let $q'=\psi(q)$, let $\pi_q:\p^6\map \p^5$ be the projection from $q$ onto a hyperplane and let $S_{14}=\pi_q(S)\subset\p^5$. The map $\psi^{-1}$ is given
by forms of degree four and the strict transform via $\psi^{-1}$ of a  line $L$ through $q'$ intersecting $W$ transversally in a point $r'$ is 
 a rational curve  $C=\psi_*^{-1}(L)$ of degree three, passing through $q$. The curve $C$  is 5-secant to $S$ because $\psi(C)=L$ and $\psi$ is given by quadratic forms. Hence  the linear span of $C$ cannot be a plane and $C$ is a twisted cubic.
The twisted cubic $C$ projects from $q$ into a 5-secant conic to $S_{14}$. 
Thus we have produced a four dimensional family of 5-secant conics to $S_{14}$, whose parameter space $\mathcal H$ is birational to the rational variety $W$.

We claim that through a general point $p\in\p^5$ there passes a unique 5-secant conic to $S_{14}$ belonging to the  family $\mathcal H$. If $L_p=<q,p>\subset\p^6$,
then $L_p\cap\Sec(S)=\{q_1,\dots, q_7\}$ and $D_p=\psi(L_p)$ is a conic because $\psi$ is given by forms of degree two and $L_p\cap S=\emptyset$. Let $\p^2_p=\langle D_p\rangle\subset\p^6$, remark that $D_p\cap W=\{\psi(q_1),\ldots, \psi(q_7)\}$ ($\psi^{-1}$ is given by forms of degree four and $\psi^{-1}_*(D_p)=L_p$), that $q'=\psi(q)\in D_p$ and that the scheme $\p^2_p\cap W$ is zero dimensional because through $p$ there pass at most finitely many
5-secant conics in the family $\mathcal H$ and no trisecant line to  $S$. Let $q'_8=(\p^2_p\cap W)\setminus\{\psi(q_1),\ldots, \psi(q_7)\}$  and let $L'_8=\langle q',q'_8\rangle$. The line $L'_8$ cuts $D_p$ in another  point $s'\neq\psi(q_i)$, $i=1,\ldots, 7$. 
Indeed, if $s'=\psi(q_i)$ for some $i=1,\ldots,7$, then $\psi_*^{-1}(L'_8)$ would be a  $3$-secant conic to $S$  passing through $q$ and
$q_i$. Hence $\pi_q(\psi_*^{-1}(L'_8))$ would be 
a $3$-secant line to $S_{14}$ passing through $p$,
which is impossible by the generality of $p$. 
Therefore  $\psi_*^{-1}(L'_8)$ is a twisted cubic, which passes through $q$ and  cuts $\langle q,p\rangle$ in the point $s=\psi^{-1}(s')$, and the conic $\pi_q(\psi_*^{-1}(L'_8))$ passes through $p$.

Any other 5-secant conic to $S_{14}$ passing through $p$ and belonging to $\mathcal H$, which is necessarily irreducible by the generality of $p$,  would determine another point of
intersection of $D_p$ (and a fortiori of $\langle D_p\rangle$) with $W$, which is impossible because $\deg(W)=8$. Thus $S_{14}$ admits a congruence of $5$-secant conics parametrized by a rational variety $\mathcal H$ birational to $S^{(2)}$. 

The irreducible component $\mathcal S_{14}$ of the Hilbert scheme parametrizing the surfaces $S_{14}\subset\p^5$ has dimension 49 and it is generically smooth. Indeed, one
can  verify that  $h^1(N_{S/\p^5})=0$ and that $h^0(N_{S/\p^5})=49$ for a general $[S]\in\mathcal S_{14}$.  
Let $h$ be the class of a hyperplane section of $X$, let $h^2$ be the class of  2-cycles $h\cdot h$ and remark that $h^2\cdot h^2=h^4=3$ and $h^2\cdot S_{14}=8$. The double point formula for $S_{14}\subset X$  (see \cite[Theorem 9.3]{Fulton}) yields $S_{14}^2=26$ and the restriction of the intersection form to  $\langle h^2, S_{14}\rangle$ has discriminant $3\cdot 26-64=14$.
Let 
$\mathcal V\subset |H^0(\mathcal O_{\p^5}(3))|=\p^{55}$ be the open
set corresponding to smooth cubic hypersurfaces.  
We have   $h^0(\mathcal I_{S_{14}}(3))=13$ so that the locus
$$ \mathbf C_{14}=\{([S],[X])\;:\; S\subset X\}\subset\mathcal S_{14}\times\mathcal V,$$ 
has dimension $49+12=61$. The image of $\pi_2:\mathbf C_{14}\to \mathcal V$ has dimension at most 54 because the general cubic does not contain any $S$ belonging to $\mathcal S_{14}$.  For every $[X]\in\mathcal \pi_2(\mathbf C_{14})$ we have
$$\dim(\pi_2^{-1}([X]))\geq \dim(\mathbf C_{14})-\dim(\pi_2(\mathbf C_{14}))=61-\dim(\pi_2(\mathbf C_{14}))\geq 61-54=7.$$
Since $h^0(N_{S/X})\geq \dim_{[S]}(\pi_2^{-1}([X]))$ for every $[S]\in \pi_2^{-1}([X])$,  to show that a general $X\in\mathcal C_{14}$ contains a surface $S_{14}$
it is sufficient to  verify that $h^0(N_{S_{14}/X})=7$ for a general $S_{14}$ and for a smooth $X\in|H^0(\mathcal I_{S_{14}}(3))|$, see also  \cite[pp.~284--285]{Nuer} for a similar argument. We verified this  via Macaulay2 and  we can conclude that a general $[X]\in\mathcal C_{14}$  contains a surface $S_{14}$ as above. 

Let
$\Phi:\p^5\map Z\subset\p^{12}$ be
given by the linear system $|H^0(\mathcal I_{S_{14}}(3))|$. The map $\Phi$  is birational onto its image $Z$ (one can   look at the resolution of the homogeneous ideal of $S_{14}$, which is generated by cubic forms, and then apply the argument in the proof \cite[Proposition 2.8]{Vermeire}; or by a direct computation). By Theorem \ref{criterion} a general $X\in\mathcal C_{14}$ is
rational being birational to $\mathcal H$. From \cite[Theorem 1]{KT} we deduce that
every $[X]\in\mathcal C_{14}$ is rational.
\end{proof}  
 Although it is well known that a general $[X]\in\mathcal C_{14}$ is rational, the previous result has been included  in order to point out a uniform approach via congruences of 5-secant conics to the rationality of general smooth cubics of discriminant equal to the  first three admissible values 14, 26 and 38.

\begin{rmk}{\rm 
The existence of the congruence of 5-secant conics was firstly discovered by studying the map $\Phi$ defined above.
The image $Z=\overline{\Phi(\p^5)}\subset\p^{12}$ is an irreducible projective variety of degree $28$ cut out by $16$ quadrics and such that
through a general point $z=\Phi(p)\in Z$, $p\in\p^5$ general, there pass 8 lines contained in $Z$. Since $S_{14}\subset\p^5$ has seven secant lines through the point $p$ by the double point formula, we deduce the existence of  a congruence of  ($3e-1$)-secant rational curves of degree $e$ to $S_{14}$. Indeed, letting  $L_1,\ldots, L_7$ be
the seven secant lines to $S_{14}$ passing through the general point $p$, they  are mapped by
$\Phi$ to seven lines $L'_1, \ldots, L'_7$ contained in $Z$ and passing through the general point $z=\Phi(p)$. Let $L'_p$ be the unique line passing through $z=\Phi(p)$ and different from the lines $L'_1, \ldots, L'_7$,
which all belong to an irreducible family of lines (a birational image of the family of secant lines to $S_{14}$). Let
$C_p\subset \p^5$ be the strict transform of $L_p'$ in $\p^5$  via $\Phi^{-1}$. Then $C_p$ is a rational curve 
of degree $e\geq 2$ passing through $p$
and cutting $S_{14}$ along a zero-dimensional scheme of degree $3e-1$ since $\Phi(C_p)=L'_p$.
One then verifies that $e=2$, see Section \ref{comp}, although this value is irrelevant for the rationality of the cubic fourfolds
through  the previous $S_{14}$.
}
\end{rmk}

\section{Rationality of cubics in \texorpdfstring{$\mathcal C_{26}$}{C26} via congruences of 5-secant conics} 

Let $S\subset\p^6$ be a septimic surface with a node, which is the projection of a smooth del Pezzo
surface of degree seven in $\p^7$ from a general point on its secant variety. Let $S_{26}\subset\p^5$ be the projection of $S$ from a general point outside the secant variety $\Sec(S)\subset\p^6$.

\begin{thm}\label{famconics26} A surface $S_{26}\subset\p^5$ as above admits a congruence of 5-secant conics parametrized by a variety birational to a rational singular cubic hypersurface in $\p^5$. Moreover,  every 
$[X]\in\mathcal C_{26}$ is rational.
\end{thm}
\begin{proof}
A septimic surface with a node $S\subset\p^6$ as above has ideal generated by seven  quadratic forms and one cubic form. The linear system $|H^0(\mathcal I_S(2))|$ defines
a Cremona transformation:
$$\psi:\p^6\map \p^6,$$
whose base locus scheme is  $S\cup P$, with $P\subset\p^6$ a plane cutting $S$ along a cubic curve.  The surface $S\cup P$ has degree 8 and arithmetic sectional genus three, and it is a projective degeneration of the base locus scheme of the  Cremona transformation considered in the proof of Theorem \ref{famconics14}

The exceptional locus 
$E\subset\p^6$ of $\psi$ is thus a hypersurface of degree seven,
whose defining polynomial is the g.c.d. of the seven degree eight polynomials defining  $\psi^{-1}\circ\psi$.
The secant variety to $S$, $\Sec(S)\subset\p^6$, is an irreducible hypersurface
of degree five by the double point formula and it is contracted by $\psi$. Therefore,   there exists a quadric hypersurface $Q\subset \p^6$ such that $E=\Sec(S)\cup Q$. Since $\psi$ is given by forms of degree two vanishing on
$S\cup P$, the projective join 
$$J(P,S)=\overline{\bigcup_{x\neq y ,\, x\in P,\, y\in S}<x,y>}\subset\p^6$$ 
is also contracted by $\psi$, so that $Q=J(P,S)$ is a rank four quadric hypersurface containing the base locus of $\psi$. 
Since $\psi^{-1}$ is given by forms of degree four, the hypersurface $E$ has points of multiplicity four along $S\cup P$ (or better, a general point of any irreducible
component is of multiplicity four for $E$).  From $E=Q\cup\Sec(S)$ and from the fact that  $Q$ is non singular at a general point of $S$ and has double points along $P$, we deduce that $\Sec(S)$ has points of multiplicity
three along $S$  and double points along $P$ (facts which can be also verified directly).
Since $\psi$ is defined by quadratic forms, the  irreducible variety  $\psi(Q)=W_2$ is contained in a hyperplane $H\subset \p^6$. The secant variety to $S$ is contracted to an irreducible variety $W_1\subset\p^6$. Clearly,  $W_1\cup W_2$ is the base locus  of $\psi^{-1}$ and we claim that $\dim(W_i)=4$ for $i=1,2$, that $\deg(W_1)=5$ and that $\deg(W_2)=3$. Indeed, a general plane $\Pi\subset\p^6$ is mapped by $\psi^{-1}$ to  a degree eight  surface $\Sigma\subset\p^6$, which is the residual intersection of four general quadric hypersurfaces through $S\cup P$. Thus  $\Sigma$ has sectional genus three and the restriction of $\psi^{-1}$ to $\Pi$ is given by a linear system of quartics with eight  base points. Moreover, $\Pi\cap H=L$ is a general line in $H$ and 
the strict transform  $\psi^{-1}_*(L)\subseteq \Sigma\cap P$ has degree one or two because  $\Sigma\subset\p^8$ has ideal generated by seven quadratic forms (see the proof of Theorem \ref{famconics14}). Hence $W_2\subset H$ is a hypersurface of degree equal to 
$4-\deg(\psi^{-1}_*(L))$, and  $W_2$  is  birational to $S\times P$ since the general fibre over $W_2$ is a line joining a point of $S$ with a point of $P$. The curve $C$ in $\Sigma\cap (S\cup P)$ has degree twelve and arithmetic genus seven. Moreover, it contains $\psi^{-1}_*(L)$
and is represented on $\Pi$ as a curve of degree seven with double points at the eight base points.
Let  $D=C-\psi^{-1}_*(L)$  be the corresponding divisor on $\Sigma$,  let $G=S\cap P$ and recall that $\deg(G)=3$. Then, taking intersections in $\Sigma$ and in $P$, we deduce $G\cdot \psi^{-1}_*(L)=D\cdot  \psi^{-1}_*(L)\in\{3,4\}$, yielding
$\deg(\psi^{-1}_*(L))=1$ and $\deg(W_2)=3$.
 Since $8=\deg(\Sigma)=\deg(\Pi\cap (W_1\cup W_2))$ and since $\deg(W_2)=3$, there are five points of intersection between $\Pi$ and $W_1$, proving $\dim(W_1)=4$ and 
$\deg(W_1)=5$ by the generality of $\Pi$.
 The rest of the proof is similar
to that of Theorem~\ref{famconics14}.

Let $q\in\p^6\setminus\Sec(S)$ be a general point, let $q'=\psi(q)$, let $\pi_q:\p^6\map\p^5$ be the projection from $q$ and let $S_{26}=\pi_q(S)\subset\p^5$. The strict transform of a line $L$ through $q'$ intersecting $W_2$ transversally in a general point $r'$ is 
a degree three irreducible curve  $C=\psi_*^{-1}(L)$ passing through $q$ and 5-secant to $S\cup P$. We claim that $C$ is 5-secant to $S$. Indeed, let   $a\geq 0$ be the number of points in $C\cap S$ and  let $b\geq 0$ be the number of points in  $C\cap P$. Since $L\cap W_1=\emptyset$ and since $L\cap W_2=r'$, we deduce from B\' ezout Theorem and from the description of the multiplicities of $Q$ and $\Sec(S)$ along $S$ and along $P$,
$\deg(C)\cdot\deg(Q)-1=5=a+2b$ and $\deg(C)\cdot\deg(\Sec(S))=15=3a+2b$. The unique solution of this system is $(a,b)=(5,0)$
\footnote{Similarly, the lines through $q'$ and a general point of $W_1$ determine twisted cubics which are 4-secant to $S$ and cut $P$ in one point because $(4,1)$ is the unique solution to $6=a+2b$ and $14=3a+2b$.}. Hence  the linear span of $C$ cannot be a plane and $C$ is a twisted cubic. This  twisted cubic  projects from $q$ into a 5-secant conic to $S_{26}$.

So far we have constructed a four dimensional family $\mathcal H$ of 5-secant conics to $S_{26}$, whose parameter space $\mathcal H$ is birational to the four dimensional singular rational cubic hypersurface $W_2$. We claim that through a general point $p\in\p^5$ there passes a unique 5-secant conic to $S_{26}$ belonging to $\mathcal H$. If $L_p=<q,p>\subset\p^6$,
then $L_p\cap\Sec(S)=\{q_1,\dots, q_5\}$, $L_p\cap Q=\{r_1,r_2\}$ and $D_p=\psi(L_p)$ is a conic. Let $\p^2_p=\langle D_p\rangle\subset\p^6$, remark that $D_p\cap W_1=\{\psi(q_1),\ldots, \psi(q_5)\}$, that
$D_p\cap W_2=\{\psi(r_1), \psi(r_2)\}$, that $q'\in D_p$ and that the scheme $\p^2_p\cap W$ is zero dimensional because through $p$ there pass at most finitely many
5-secant conics in the family $\mathcal H$ and no trisecant line to  $S$. Let $r'_3=(\p^2_p\cap W_2)\setminus\{\psi(r_1), \psi(r_2)\}$ and let $L'_3=\langle q',r'_3\rangle$. We claim that the line $L'_3$ cuts $D_p$ in a point $s'\in D_p\setminus W_2$. Indeed, since  $q'\not\in H$, $L'_3\cap H=r'_3=L'_3\cap W_2$. Then  $\psi_*^{-1}(L'_3)$ is a twisted cubic, which  cuts $\langle q,p\rangle$ in the point $s=\psi^{-1}(s')$ and the conic $\pi_q(\psi^{-1}(L'_3))$ passes through $p$.

Any other 5-secant conic to $S_{26}$ passing through $p$ and belonging to $\mathcal H$,  necessarily irreducible by the generality of $p$,  would determine another point of
intersection of $D_p$ (and a fortiori of $\langle D_p\rangle$) with $W_2$, which is impossible because $\deg(W_2)=3$.  In conclusion, 
$S_{26}$ admits a congruence of $5$-secant conics birationally parametrized by $W_2$.

 We claim that a general $[X]\in\mathcal C_{26}$ contains a $S_{26}$ as above. The irreducible component of the Hilbert scheme parametrizing these surfaces $S_{26}$  has dimension 42 and it is generically smooth.
Using that $S_{26}^2=25$ (which follows from the double point formula for surfaces with nodes, see \cite[Theorem 9.3]{Fulton}) and that $h^0(\mathcal I_{S_{26}}(3))=14$,  by reasoning exactly as in the last part of the proof of Theorem \ref{famconics14} to prove the claim it suffices to produce a general $S_{26}\subset\p^5$ and a smooth  $X\in|H^0(\mathcal I_{S_{26}}(3))|$ such  that $h^0(N_{S_{26}/X})=1=42+13-54$.  An explicit  computational verification, see for example the end of Section \ref{comp}, shows that this holds, proving that a general $[X]\in\mathcal C_{26}$ contains a surface $S_{26}$.

The homogeneous ideal of $S_{26}\subset\p^5$ is generated by $14$ cubic forms defining a map 
$\Phi:\p^5\map\p^{13} $, which is 
birational onto its image  $Z=\overline{\phi(\p^5)}\subset\p^{13}$. By Theorem \ref{criterion} a general $X\in\mathcal C_{26}$  is rational being birational to $\mathcal H$ and hence to $W_2$. From \cite[Theorem 1]{KT} we deduce that
every $[X]\in\mathcal C_{26}$ is rational.
\end{proof}

\begin{rmk}\label{conicsR}{\rm We firstly discovered the congruence of 5-secant conics to the surfaces $S_{26}$ described above via the map $\Phi$ defined above.
The variety $Z\subset\p^{13}$ has degree $34$,
it cut out by $20$ quadrics,
and  through a general point $z\in Z$ there pass $6$ lines contained in $Z$, see Section \ref{comp}.
By the double point formula, we deduce that through a general point $p\in\p^5$ there pass
five secant lines to $S_{26}$. The strict transforms via $\Phi^{-1}$ of the lines through a general point, not coming from secant lines to $S_{26}$,  define a congruence of (3$e$--1)-secant curves to $S_{26}$ of degree $e\geq 2$. By a standard computation, one sees that
$e=2$ although the value of $e$ is insignificant for the rationality
conclusions. }\end{rmk}
\begin{rmk}{\rm
The rationality of a general $[X]\in \mathcal C_{26}$ can be also proved by taking as $S_{26}\subset\p^5$ the rational septimic
scroll with three nodes recently considered by Farkas and Verra in \cite{FV}. Also these surfaces admit a congruence of
5-secant conics, whose parameter space is rational. Moreover,  the associated map $\Phi$ is birational onto its image.
}\end{rmk}

\section{Rationality of cubics in \texorpdfstring{$\mathcal C_{38}$}{C38} via congruences of 5-secant conics} 

Let $S_{38}\subset\p^5$ be a general  degree 10 smooth surface of sectional genus 6 obtained as the image
of $\p^2$ by the linear system of plane curves of degree 10 having 10 fixed triple points, which were studied classically
by Coble. 

\begin{thm}\label{famconics38} A general surface $S_{38}\subset\p^5$ admits a congruence of 5-secant conics. Moreover,  every
$[X]\in\mathcal C_{38}$ is rational.
\end{thm}
\begin{proof}
As shown by Nuer in \cite{Nuer},  these surfaces are contained in a general $[X]\in\mathcal C_{38}$. Moreover $h^0(\mathcal I_{S_{38}}(2))=0=h^1(\mathcal I_{S_{38}}(3))$ and 
the homogeneous ideal of $S_{38}$ is generated by $10$ cubic forms, whose first syzygies are generated by the linear ones.  Via the argument in the proof of  \cite[Proposition 2.8]{Vermeire}, we deduce that the linear system 
$|H^0(\mathcal I_{S_{38}}(3))|$ defines a birational map onto the image
$\Phi:\p^5\map Z\subset\p^9.$
An explicit calculation with Macaulay2, see Section \ref{comp}, shows that 
the image $Z=\overline{\Phi(\p^5)}\subset\p^{9}$ is an irreducible variety of degree $20$ cut out by $16$ cubics and such that
through a general point
$z=\phi(p)\in Z$ there pass $8$ lines contained in $Z$. 
Since $S_{38}\subset\p^5$ has seven secant lines
passing through a general point $p\in\p^5$, we deduce the existence of a congruence of ($3e-1$)-secant rational curves of degree $e$
to $S_{38}$. We have, once again, $e=2$ although this is irrelevant for the rationality conclusions.

There exist singular rational cubic hypersurfaces through a general $S_{38}$ with a unique ordinary double point,
as one can directly verify by using Macaulay2. Theorem \ref{criterion} yields the rationality of the parameter space $\mathcal H$ of the congruence of 5-secant conics to $S_{38}$. Then every cubic through $S_{38}$ (with at most rational singularities) is rational by Theorem \ref{criterion} and a general cubic $[X]\in\mathcal C_{38}$ is then rational. From \cite[Theorem 1]{KT} we deduce that
every $[X]\in\mathcal C_{38}$ is rational.
\end{proof}

\begin{rmk}\label{C38sing}{\rm 
 The irreducible boundary
cubics of $\mathcal C_{38}$ corresponding to singular elements in $|H^0(\mathcal I_{S_{38}}(3))|$ are all rational by Theorem \ref{criterion}. 
These rational cubic hypersurfaces with a rational double point also lie on the boundary of $\mathcal C_{44}$, see also \cite[p.~285]{Nuer}.

Let $S_d\subset\p^5$ with $d=44-6j$ and with $j=0,\ldots, 6$ be the smooth surfaces constructed by Nuer to describe  
$\mathcal C_d$ and which are flat projective  deformations of $S_{38}\subset\p^5$ (although belonging to different irreducible components of the Hilbert scheme), see \cite[p.~286]{Nuer}. 
In Table~\ref{pentasecanti} below, we give 
the number of $5$-secant conics to $S_d$ passing through a general point of $\p^5$.
}
\end{rmk}

\begin{table}[htbp]
\centering
\tabcolsep=4.7pt 
\begin{tabular}{lclll}
\hline
d  & Surface $S\subset\PP^5$ &  $2$-secant lines  & $5$-secant conics &   Multidegree \\  
\hline \hline 
8 & \scriptsize{\begin{tabular}{c} Image of the plane via the linear system \\ of quintic curves with $15$ general base points \end{tabular}} & $12$  & $6$ &  $1, 3, 9, 17, 21, 15$   \\
\hline
14 & \scriptsize{\begin{tabular}{c} Image of the plane via the linear system \\ of sextic curves with $10$ general simple \\ base points and $4$ general double points \end{tabular}} & $11$  & $5$ &  $1, 3, 9, 17, 21, 16$   \\
\hline
20 & \scriptsize{\begin{tabular}{c} Image of the plane via the linear system \\ of septic curves with $6$ general simple \\ base points, $6$ general double points \\ and one general triple point \end{tabular}} & $10$  & $4$ &  $1, 3, 9, 17, 21, 17$   \\
\hline
26 & \scriptsize{\begin{tabular}{c} Image of the plane via the linear system \\ of septic curves with $3$ general simple \\ base points and $9$ general double points \end{tabular}} & $9$  & $3$ &  $1, 3, 9, 17, 21, 18$   \\
\hline
32 & \scriptsize{\begin{tabular}{c} Image of the plane via the linear system \\ of curves of degree $9$ with one general \\ simple base point,  $4$ general double points \\ and $6$ general triple points \end{tabular}} & $8$  & $2$ &  $1, 3, 9, 17, 21, 19$   \\
\hline
38 & \scriptsize{\begin{tabular}{c} Image of the plane via the linear system \\ of curves of degree $10$ with $10$ general \\ triple points \end{tabular}} & $7$  & $1$ &  $1, 3, 9, 17, 21, 20$   \\
\hline
44 & \scriptsize{\begin{tabular}{c} Fano embedded Enriques surface  \end{tabular}} & $6$  & $0$ &  $1, 3, 9, 17, 21, 21$   \\
\hline
\end{tabular}
 \caption{Smooth
 surfaces $S\subset\PP^5$ 
 of degree $10$ and sectional genus $6$ 
 cut out by $10$ cubics and contained in the 
 generic cubic fourfold of $\mathcal{C}_d$. 
  The $3^{\mathrm{rd}}$ and $4^{\mathrm{th}}$ columns 
 contain, respectively, 
 the number of $2$-secant lines and the number 
 of $5$-secant conics
 to $S$ passing through a general point of $\PP^5$. 
 The $5^{\mathrm{th}}$ column contains the multidegree 
 of the map defined by $|H^0(\mathcal{I}_{S}(3))|$.} 
\label{pentasecanti} 
\end{table}

\section{Congruences of 5-secant conics from a computational point of view}\label{comp}

In this section we illustrate 
how one can detect the $5$-secant conic congruence property 
in specific examples using 
the computer algebra system Macaulay2 \cite{macaulay2}.

We start by loading the file \texttt{Cubics.m2}, 
 which is available as an ancillary file to our arXiv submission.
This file contains equations for three explicit examples related 
to the three divisors $\mathcal{C}_d$, $d=14,26,38$, 
but here, for brevity, we will consider only the case of $\mathcal{C}_{26}$.
After loading the file, the tools  for working with rational maps
provided by the \texttt{Cremona} package (included with Macaulay2) 
will be also available.  
{\footnotesize
\begin{Verbatim}[commandchars=&\[\]]
Macaulay2, version 1.11
with packages: &colore[airforceblue][ConwayPolynomials], &colore[airforceblue][Elimination], &colore[airforceblue][IntegralClosure], &colore[airforceblue][InverseSystems], 
               &colore[airforceblue][LLLBases], &colore[airforceblue][PrimaryDecomposition], &colore[airforceblue][ReesAlgebra], &colore[airforceblue][TangentCone]
&colore[darkorange][i1 :] &colore[airforceblue][needsPackage] "&colore[bleudefrance][Cubics]"; 
\end{Verbatim}
} \noindent 
The next line of code produces two rational maps,  $f:\PP^2\dashrightarrow\PP^5$ and 
$\varphi:\PP^5\dashrightarrow\PP^{13}$, which are birational parameterizations 
of a del Pezzo surface of degree $7$ with a node
 $S=S_{26}\subset\PP^5$ and of a subvariety $Z\subset\PP^{13}$, respectively;
here \texttt{QQ} specifies that we want the ground field to be $\mathbb{Q}$.
The base locus of $\varphi$ is exactly  the (closure of the) image of $f$, 
and the information on the projective degrees
says  that $\varphi$ is birational onto a degree $34$ variety.
{\footnotesize
\begin{Verbatim}[commandchars=&\[\]] 
&colore[darkorange][i2 :] &colore[darkorchid][time] (f,phi) = &colore[bleudefrance][example](26,&colore[darkspringgreen][QQ]);
     &colore[Sepia][-- used 0.630698 seconds] 
&colore[darkorange][i3 :] f;
o3 : RationalMap (cubic rational map from PP^2 to PP^5)
&colore[darkorange][i4 :] phi;
o4 : RationalMap (cubic rational map from PP^5 to PP^13)
&colore[darkorange][i5 :] &colore[darkorchid][time] (S = &colore[airforceblue][image] f, Z = &colore[airforceblue][image] phi);
     &colore[Sepia][-- used 0.000034454 seconds]
&colore[darkorange][i6 :] &colore[darkorchid][time] (&colore[airforceblue][degree] phi, &colore[airforceblue][multidegree] phi)
     &colore[Sepia][-- used 0.000026732 seconds]
o6 = (1, {1, 3, 9, 20, 32, 34})
\end{Verbatim}
} \noindent 
Note that the code above does not make any  computation as the data are stored internally, but 
it is not difficult to perform this kind of computations (for instance 
$Z$ can be quickly obtained with \texttt{image(2,phi)}, see \cite{Sta18}).
Now we choose a random point\footnote{Taking 
the ground field to be the pure trascendental extension $\mathbb{Q}(a_0,\ldots,a_5)$,
one could also take the most generic point of $\PP^5$,
i.e. the symbolic point $(a_0,\ldots,a_5)$.
This, however, would considerably increase the running time.}
$p\in\PP^5$
and compute the locus $E\subset\PP^5$ consisting 
of the union of all secant lines to $S$ passing through $p$,
and the locus $V\subset Z\subset \PP^{12}$ consisting of the 
union of all lines contained in $Z$ and passing through $\varphi(p)\in Z$.
The former is obtained using standard elimination techniques, 
while for the latter we use the procedure given in \cite[p.~57]{libro}.
{\footnotesize
\begin{Verbatim}[commandchars=&!$]  
&colore!darkorange$!i7 :$ p = &colore!darkorchid$!for$ i &colore!darkorchid$!to$ 5 &colore!darkorchid$!list$ &colore!airforceblue$!random$(-100,100)  
o7 = {79, 16, -34, -12, -40, -62}
&colore!darkorange$!i8 :$ &colore!darkorchid$!time$ (E,V) = (&colore!bleudefrance$!secantCone$(p,f),&colore!bleudefrance$!coneOfLines$(p,phi)); 
     &colore!Sepia$!-- used 0.684973 seconds$
&colore!darkorange$!i9 :$ (&colore!airforceblue$!dim$ E - 1,&colore!airforceblue$!degree$ E)   &colore!Sepia$!-- dimension and degree of E$
o9 = (1, 5)
&colore!darkorange$!i10 :$ (&colore!airforceblue$!dim$ V - 1,&colore!airforceblue$!degree$ V)  &colore!Sepia$!-- dimension and degree of V$
o10 = (1, 6)
\end{Verbatim}  
} \noindent
Five of the six lines in $V$ 
come from lines in $E$. So, the extra line $L=\overline{V\setminus \varphi(E)}$ can be determined
with a saturation computation.
{\footnotesize
\begin{Verbatim}[commandchars=&!$]  
&colore!darkorange$!i11 :$ &colore!darkorchid$!time$ L = &colore!airforceblue$!saturate$(V,phi E);
     &colore!Sepia$!-- used 0.640101 seconds$ 
&colore!darkorange$!i12 :$ (&colore!airforceblue$!dim$ L - 1,&colore!airforceblue$!degree$ L)  &colore!Sepia$!-- dimension and degree of L$
o12 = (1, 1)
\end{Verbatim}  
} \noindent
The $5$-secant conic to $S$ passing through $p$ 
is the inverse image of the extra line $L$.
{\footnotesize
\begin{Verbatim}[commandchars=&!$]  
&colore!darkorange$!i13 :$ &colore!darkorchid$!time$ C = phi^* L;
     &colore!Sepia$!-- used 0.0256889 seconds$ 
&colore!darkorange$!i14 :$ (&colore!airforceblue$!dim$ C - 1,&colore!airforceblue$!degree$ C)
o14 = (1, 2)
&colore!darkorange$!i15 :$ (&colore!airforceblue$!dim$(C+S) - 1,&colore!airforceblue$!degree$(C+S))
o15 = (0, 5)
\end{Verbatim}  
} \noindent
The conic above could also be determined with a single command.
{\footnotesize
\begin{Verbatim}[commandchars=&!$]  
&colore!darkorange$!i16 :$ &colore!darkorchid$!time$ C == &colore!bleudefrance$!fiveSecantConic$(p,f,phi)  
     &colore!Sepia$!-- used 1.57295 seconds$
o16 = true
\end{Verbatim}  
} \noindent
Finally, in the following code, we take a smooth cubic hypersurface $X\subset\mathbb{P}^5$ containing the surface $S$ 
and compute $h^0({N}_{S/X})$, where ${N}_{S/X}$ is the normal sheaf of $S$ in $X$.
{\footnotesize
\begin{Verbatim}[commandchars=&!$]  
&colore!darkorange$!i17 :$ &colore!darkorchid$!time$ X = &colore!airforceblue$!ideal$ &colore!bleudefrance$!randomSmoothCubic$ S; 
     &colore!Sepia$!-- used 0.121729 seconds$ 
&colore!darkorange$!i18 :$ &colore!darkorchid$!time$ &colore!airforceblue$!rank$ &colore!darkspringgreen$!HH$^0 &colore!bleudefrance$!normalSheaf$(S,X) 
     &colore!Sepia$!-- used 61.1055 seconds$
o18 = 1
\end{Verbatim}  
} \noindent

\providecommand{\bysame}{\leavevmode\hbox to3em{\hrulefill}\thinspace}
\providecommand{\MR}{\relax\ifhmode\unskip\space\fi MR }
\providecommand{\MRhref}[2]{%
  \href{http://www.ams.org/mathscinet-getitem?mr=#1}{#2}
}
\providecommand{\href}[2]{#2}

\end{document}